\documentclass{amsart}

\usepackage{graphicx}
\usepackage{tikz}
\usetikzlibrary{matrix,arrows,decorations.pathmorphing}
\usepackage{tikz-cd}
\usepackage{hyperref}
\usepackage{amsthm}
\usepackage[mathscr]{euscript}
\usepackage{amssymb}
\usepackage{enumitem}
\newtheorem{theorem}{Theorem}[section]
\newtheorem{lemma}[theorem]{Lemma}
\theoremstyle{definition}

\newtheorem{conjecture}[theorem]{Conjecture}
\newtheorem{convention}[theorem]{Convention}

\newtheorem{proposition}[theorem]{Proposition}

\theoremstyle{remark}

\numberwithin{equation}{section}

\newcommand{\Z}{\mathbb{Z}}
\newcommand{\opn}{\operatorname}
\newcommand{\wt}{\widetilde}
\newcommand{\RomanNumeralCaps}[1]
    {\MakeUppercase{\romannumeral #1}}

\begin{document}

\title{The topological period-index problem over $8$-complexes, \RomanNumeralCaps{2}}

\author{Xing Gu}
\address{Max Plank Institute for Mathematics, Vivatsgasse 7, 53111 Bonn, Germany}

\address{School of Mathematics and Statistics, the University of Melbourne, Parkville VIC 3010, Australia}
\curraddr{}
\email{gux2006@mpim-bonn.mpg.de}
\thanks{The bulk of this work was completed at the University of Melbourne, Australia, and further revisions are made at the Max Planck Institute for Mathematics, Germany. The author is grateful to the Australian Research Council, the University of Melbourne and the Max Planck Institute for Mathematics for their supports.}



\date{}

\dedicatory{}


\begin{abstract}
We complete the study of the topological period-index problem over $8$ dimensional connected finite CW complexes started in a preceding paper. More precisely, we determine the sharp upper bound of the index of a topological Brauer class $\alpha\in H^3(X;\Z)$, where $X$ is of the homotopy type of an $8$ dimensional finite CW complex and the period of $\alpha$ is divisible by $4$.
\end{abstract}
\maketitle
\section{introduction}
This paper is a sequel to Gu \cite{Gu}, where we investigated the topological period-index problem (TPIP) over finite CW complexes of dimension $8$, and determined the upper bound of indices of topological Brauer classes with period $n$ not divisible by $4$. In this paper we give a complete answer to the TPIP over finite $8$-complexes by studying the case $4|n$.

For a path-connected topological space $X$, let $\operatorname{Br}(X)$ be the \emph{topological Brauer group} defined in \cite{An1}, whose underlying set is the Azumaya algebras (i.e. bundles of complex matrix algebras over $X$) modulo the \emph{Brauer equivalence}: $\mathscr{A}_0$ and $\mathscr{A}_1$ are called \emph{Brauer equivalent} if there are vector bundles $\mathscr{E}_0$ and $\mathscr{E}_1$ such that
$$\mathscr{A}_0\otimes\operatorname{End}(\mathscr{E}_0)\cong \mathscr{A}_1\otimes\operatorname{End}(\mathscr{E}_1).$$
The multiplication is given by tensor product. Sometimes we call an element of $\operatorname{Br}(X)$ a \emph{topological Brauer class} of $X$.

Azumaya algebras over $X$ of degree $r$ are classified up to isomorphism by the collection of isomorphism classes of principal $PU_r$ bundles over $X$, i.e., the cohomology set $H^{1}(X;PU_r)$, where $PU_r$ is the projective unitary group of degree $r$. Consider the short exact sequence of Lie groups
\begin{equation}\label{ses1}
1\rightarrow S^1\rightarrow U_r\rightarrow PU_r\rightarrow 1
\end{equation}
where the arrow $S^1\rightarrow U_r$ is the inclusion of scalars. Then the connecting homomorphism
\begin{equation}\label{Bockstein}
H^{1}(X;PU_r)\rightarrow H^{2}(X;S^1)\cong H^{3}(X;\mathbb{Z})
\end{equation}
associates an Azumaya algebra $\mathscr{A}$ to a class $\alpha\in H^{3}(X;\mathbb{Z})$. The exactness of the sequence (\ref{ses1}) implies that
\begin{enumerate}
\item $\alpha\in H^{3}(X;\mathbb{Z})_{\textrm{tor}}$, the subgroup of torsion elements of $ H^{3}(X;\mathbb{Z})$, and
\item the class $\alpha$ only depends on the Brauer equivalence class of $\mathscr{A}$.
\end{enumerate}
An Azumaya algebra associated to a Brauer class may alternatively be described as follows. Let $\mathscr{A}$ be an Azumaya algebra of degree $r$ over a finite CW complex $X$. As discussed above, one can associate a principal $PU_r$-bundle, i.e., a homotopy class of maps $X\rightarrow\mathbf{B}PU_r$, where $\mathbf{B}PU_r$ is the classifying space of $PU_r$. An elementary computation yields $H^3(\mathbf{B}PU_r;\Z)\cong\Z/r$. The Azumaya algebra $\mathscr{A}$ is associated to a class $\alpha\in H^{3}(X;\mathbb{Z})$ if and only if the following homotopy commutative diagram
\begin{equation}\label{eq:lift}
\begin{tikzcd}
&\mathbf{B}PU_r\arrow{d}\\
X\arrow[ur,dashed]\arrow{r}{\alpha}&K(\mathbb{Z},3)
\end{tikzcd}
\end{equation}
exists such that the dashed arrow represents the (isomorphism class) $PU_r$-bundle classifying $\mathscr{A}$, the vertical arrow corresponds to a generator of $H^3(\mathbf{B}PU_r)\cong\Z/r$, and the bottom arrow corresponds to the cohomology class $\alpha$.

Therefore, there is a homomorphism $\textrm{Br}(X)\rightarrow H^{3}(X;\mathbb{Z})_{\textrm{tor}}$, and we call $H^{3}(X;\mathbb{Z})_{\textrm{tor}}$ the cohomological Brauer group of $X$, and sometimes denote it by $\textrm{Br}'(X)$. Serre \cite{Gr} showed that when $X$ is a finite CW complex, this homomorphism is an isomorphism. Hence, for any $\alpha\in H^{3}(X;\mathbb{Z})_{\textrm{tor}}$, there is some $r$ such that a $PU_r$-torsor over $X$ is associated to $\alpha$ via the homomorphism (\ref{Bockstein}).

Let $\operatorname{per}(\alpha)$ denote the order of $\alpha$ as an element of the group $H^{3}(X;\mathbb{Z})$. Serre \cite{Gr} also showed $\operatorname{per}(\alpha)|r$ for all $r$ such that there is a $PU_r$-torsor over $X$ associated to $\alpha$ in the way described above. Let $\operatorname{ind}(\alpha)$ denote the greatest common divisor of all such $r$, then in particular we have
\begin{equation}\label{per div ind}
\operatorname{per}(\alpha)|\operatorname{ind}(\alpha).
\end{equation}

The preceding definitions are motivated by their algebraic analogs. We refer to the introduction of \cite{An} for the algebraic version of the definitions as well as more algebraic backgrounds. The period-index conjecture, stated as follows, plays a key role in the study of Brauer groups:

\begin{conjecture}[\cite{Co}, Colliot-Th{\'e}l{\`e}ne]\label{algconj}
Let $k$ be either a $C_d$-field or the function field of a $d$-dimensional variety over an algebraically closed field. Let $\alpha\in \textrm{Br}(k)$, and suppose that $\operatorname{per}(\alpha)$ is prime to the characteristic of $k$. Then
\[\operatorname{ind}(\alpha)|\operatorname{per}(\alpha)^{d-1}.\]
\end{conjecture}

Inspired by and to shed light on the above conjecture, in \cite{An}, Antieau and Williams proposed a topological analog of Conjecture \ref{algconj}, as follows:

\textit{For a given Bruaer class $\alpha$ of a finite CW complexes $X$ of dimension $2d$, what can be said about the integers $\operatorname{ind}(\alpha)$ and $\operatorname{per}(\alpha)^{d-1}$? More specifically, does the former divides the latter?}

Antieau and Williams in \cite{An} give a negative answer to the above question in the case $d=3$ by showing that $\operatorname{ind}(\alpha)$ may be as great as $2\operatorname{per}(\alpha)^2$, depenting on the parity of $n$. In \cite{Cr} Crowley and Grant give a positive answer in the case $d=3$ and $X$ is a $Spin^c$ manifold. In \cite{Gu} the author shows that for $d=4$ the answer is negative and  that $\operatorname{ind}(\alpha)$ may be as great as $6\operatorname{per}(\alpha)^3$, depending on the parity of $n$, as well as whether $3|n$. (See Theorem \ref{last} below.) In this paper we continue our exploration in this direction.

For further explanations of the preceding definitions and backgrounds on the topological period-index problem, see \cite{An}, \cite{An1}, \cite{Cr} and \cite{Gu}. All definitions and notations in this paper are consistent with those in \cite{Gu}. More precisely, the expression $\epsilon_p(n)$ denotes the greatest common divisor of a prime $p$ and a positive integer $n$, and we let $\beta_n$ be the canonical generator of $H^3(K(\mathbb{Z}/n,2);\mathbb{Z})$, i.e., the image of the identity class in $H^2(K(\mathbb{Z}/n,2);\mathbb{Z}/n)$ under the Bockstein homomorphism. 

In \cite{Gu} we have shown the following

\begin{theorem}[Theorem 1.6, \cite{Gu}]\label{last}
Let $X$ be a topological space of homotopy type of an $8$-dimensional connected finite CW-complex, and let $\alpha\in H^{3}(X;\mathbb{Z})_{\operatorname{tor}}$ be a topological Brauer class of period $n$. Then
\begin{equation}\label{bound}
\operatorname{ind}(\alpha)|\epsilon_{2}(n)\epsilon_{3}(n)n^3.
\end{equation}
In addition, if $X$ is the $8$-skeleton of $K(\mathbb{Z}/n,2)$, and $\alpha$ is the restriction of the fundamental class $\beta_n\in H^{3}(K(\mathbb{Z}/n,2),\mathbb{Z})$, then
\begin{equation*}
\begin{cases}
\operatorname{ind}(\alpha)=\epsilon_{2}(n)\epsilon_{3}(n)n^3,\quad\textrm{$4\nmid n$,}\\
\epsilon_{3}(n)n^3|\operatorname{ind}(\alpha), \quad\textrm{$4|n$.}
\end{cases}
\end{equation*}
In particular, the sharp lower bound of $e$ such that $\operatorname{ind}(\alpha)|n^{e}$ for all $X$ and $\alpha$ is $4$.
\end{theorem}

The goal of this paper is to improve Theorem \ref{last} and show the following

\begin{theorem}\label{main}
Let $X$ be a topological space of homotopy type of an $8$-dimensional connected finite CW-complex, and let $\alpha\in H^{3}(X;\mathbb{Z})_{\operatorname{tor}}$ be a topological Brauer class of period $n$. Then
\begin{equation}\label{bound1}
\begin{cases}
\operatorname{ind}(\alpha)|2\epsilon_{3}(n)n^3,\textrm{ if }n\equiv 2\pmod{4},\\
\operatorname{ind}(\alpha)|\epsilon_{3}(n)n^3,\textrm{ otherwise}.
\end{cases}
\end{equation}
In addition, if $X$ is the $8$-skeleton of $K(\mathbb{Z}/n,2)$ and $\alpha$ is the restriction of the fundamental class $\beta_n\in H^{3}(K(\mathbb{Z}/n,2),\mathbb{Z})$, then
\begin{equation*}
\begin{cases}
\operatorname{ind}(\alpha)=2\epsilon_{3}(n)n^3,\textrm{ if }n\equiv 2\pmod{4},\\
\operatorname{ind}(\alpha)=\epsilon_{3}(n)n^3,\textrm{ otherwise}.
\end{cases}
\end{equation*}
In particular, the sharp lower bound of $e$ such that $\operatorname{ind}(\alpha)|n^{e}$ for all $X$ and $\alpha$ is $4$.
\end{theorem}

One readily sees that the only thing remains to show is the following
\begin{proposition}\label{mainprop}
Let $X$ be the $8$-skeleton of $K(\mathbb{Z}/n,2)$ and $\alpha\in H^3(X;\Z)$ the restriction of $\beta_n$. When $4|n$, we have
\[\operatorname{ind}(\alpha)=\epsilon_3(n)n^3.\]
\end{proposition}

Theorem \ref{main} may be interpreted in terms of the twisted complex $K$-theory and the twisted Atiyah-Hirzebruch spectral sequences (AHSS). For a CW complex $X$ and a class $\alpha\in H^3(X;\Z)$, the twisted complex $K$-theory associates the pair $(X,\alpha)$ to a graded ring $K^*(X)_{\alpha}$ in a contravariant manner. For $\alpha=0$, it is the usual complex $K$-theory $K^*(X)$. The readers may refer to \cite{At} and \cite{At1} for more backgrounds.

As for the usual complex $K$-theory, there is a twisted AHSS of the pair $(X,\alpha)$ converging to the twisted K-theory $K^*(X)_{\alpha}$. We denote it by $(\wt{E}_*^{*,*}, \wt{d}_*^{*,*})$. The spectral sequence satisfies
\[ \wt{E}_2^{s,t}\cong H^s(X;K^t(\opn{pt}))\]
where $K^*$ denotes the complex topological $K$-theory in the usual sense. By the Bott Periodicity Theorem, we have
\begin{equation}\label{AH}
 \wt{E}_2^{s,t}\cong
 \begin{cases}
 H^s(X;\mathbb{Z}), t \textrm{ even},\\
 0, t \textrm{ odd}.
 \end{cases}
\end{equation}
The twisted AHSS is considered in \cite{An}, \cite{An1}, \cite{An2} and \cite{Gu}. The following result can be derived immediately from \cite{An1}.

\begin{theorem}[Theorem 3.1, \cite{Gu}]\label{AH diff}
Let $X$ be a connected finite CW-complex and let $\alpha\in\textrm{Br}(X)$. Consider $\wt{E}_{*}^{*,*}$, the twisted Atiyah-Hirzebruch spectral sequence of the pair $(X,\alpha)$ with differentials $\wt{d}_{r}^{s,t}$ with bi-degree $(r, -r+1)$. In particular, $\wt{E}_{2}^{0,0}\cong\mathbb{Z}$, and any $\wt{E}_{r}^{0,0}$ with $r>2$ is a subgroup of $\mathbb{Z}$ and therefore generated by a positive integer. The subgroup $\wt{E}_{3}^{0,0}$ (resp.  $\wt{E}_{\infty}^{0,0}$) is generated by $\operatorname{per}(\alpha)$ (resp. $\operatorname{ind}(\alpha))$.
\end{theorem}

Following Theorem \ref{AH diff}, it can be easily deduced from Theorem B of \cite{An} and Theorem \ref{last} that when $X=K(\mathbb{Z}/n,2)$ and $\alpha=\beta_n$, the differentials $\wt{d}_r^{0,0}$ are surjective onto entries on the $E_2$-page for $r<7$, and when $4\nmid n$, also for $r=7$. Theorem \ref{main}, however, provides the first known example of a $\wt{d}_r^{0,0}$ that is NOT surjective onto some $\wt{E}_2^{s,t}$.
\begin{proposition}
Suppose $4|n$. In the twisted Atiyah-Hirzebruch spectral sequence $(\wt{E}_*^{*,*}, \wt{d}_*^{*,*})$ of the pair $(K(\mathbb{Z}/n,2),\beta_n)$, the image of $\wt{d}_7^{0,0}$ is $2\wt{E}_2^{7,-6}$.
\end{proposition}
\begin{proof}[Proof assuming Theorem \ref{main}]
This is easily deduced from the preceeding paragraph and Theorem \ref{main}, once we recall the cohomology of $K(\mathbb{Z}/n,2)$, as in Section 2. See Figure \ref{twisted AH}.
\end{proof}

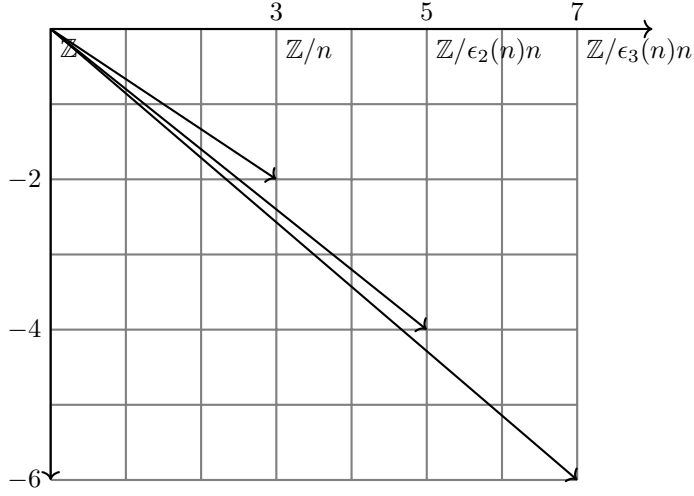
\begin{figure}[!]
\begin{tikzpicture}
\draw[step=1cm,gray, thick] (0,0) grid (7,-6);
\draw [ thick, <->] (0,-6)--(0,0)--(8,0);

\node [above] at (3,0) {$3$};
\node [above] at (5,0) {$5$};
\node [above] at (7,0) {$7$};
\node [left] at (0,-2) {$-2$};
\node [left] at (0,-4) {$-4$};
\node [left] at (0,-6) {$-6$};

\node [below right] at (0,0) {$\mathbb{Z}$};
\node [below right] at (3,0) {$\mathbb{Z}/n$};
\node [below right] at (5,0) {$\mathbb{Z}/\epsilon_{2}(n)n$};
\node [below right] at (7,0) {$\mathbb{Z}/\epsilon_{3}(n)n$};
\draw (0,0) [thick, ->] to (3,-2);
\draw (0,0) [thick, ->] to (5,-4);
\draw (0,0) [thick, ->] to (7,-6);
\end{tikzpicture}
\caption{The twisted Atiyah-Hirzebruch spectral sequence associated to $K(\mathbb{Z}/n,2)$ and $\beta_n$. }\label{twisted AH}
\end{figure}

The technical core of this paper is to solve a homotopy lifting problem of the same nature as the one displayed in (\ref{eq:lift}), with some modifications for technical convenience. We recall some more notations from \cite{Gu}. Let $m,n$ be integers. Then $\mathbb{Z}/n$ is a closed normal subgroup of $SU_{mn}$ in the sense of the following monomorphism of Lie groups:
\begin{equation*}
\mathbb{Z}/n\hookrightarrow SU_{mn}: t\mapsto e^{2\pi it/n}\mathbf{I}_{mn},
\end{equation*}
where $\mathbf{I}_r$ is the identity matrix of degree $r$. We denote the quotient group by $P(n,mn)$, and consider its classifying space $\mathbf{B}P(n,mn)$. It follows immediately from Bott's periodicity theorem that we have
\begin{equation}\label{Bott}
\pi_{i}(\mathbf{B}P(n,mn))\cong
\begin{cases}
\mathbb{Z}/n,\quad\textrm{if $i=2$},\\
\mathbb{Z},\quad\textrm{if $2<i<2mn+1$, and $i$ is even,}\\
0, \quad\textrm{if $0<i<2mn$, and $i$ is odd.}\\
\end{cases}
\end{equation}
In the context of this paper, the space $\mathbf{B}P(n,mn)$ may be considered as a refinement of $\mathbf{B}PU_{mn}$.  For a finite CW-complex $X$ and a topological Brauer class $\alpha$ of period $n$, there is an element $\alpha'$ in $H^2(X;\mathbb{Z}/n)$, which is sent to $\alpha$ by the Bockstein homomorphism. Consider the following lifting diagram
\begin{equation}\label{lift'}
\begin{tikzcd}
&\mathbf{B}P(n,mn)\arrow{d}\\
X\arrow[ur,dashed]\arrow{r}{\alpha'}&K(\mathbb{Z}/n,2)
\end{tikzcd}
\end{equation}
where the vertical arrow is the canonical map from $\mathbf{B}P(n,mn)$ to the bottom stage of its Postnikov tower.
We have the following
\begin{proposition}\label{lift to BP}[Proposition 4.3, \cite{Gu}]
Let $X$, $\alpha$ be as above, and suppose that $H^2(X;\mathbb{Z})=0$. Then $\alpha$ is classified by an Azumaya algebra of degree $mn$ if and only if the lifting in diagram (\ref{lift'}) exists.
\end{proposition}

In Section 2 we recall some facts on Eilenberg-Mac Lane spaces. Section 3 is a collection of lemmas on the cohomology of $\mathbf{B}P(n,mn)$. In Section 4 we consider a k-invariant in the Postnikov decomposition of $\mathbf{B}P(n,mn)$, for suitable $n$ and $m$, leading to the proof of Proposition \ref{mainprop} and Theorem \ref{main}.

The author is grateful to Benjamin Antieau for introducing the topological period-index problem to the author, and to Diarmuid Crowley and Christian Haesemeyer for helpful conversations on this work.

\section{recollection of facts on eilenberg - mac lane spaces}
All the facts recalled here are either well-known or easily deduced from \cite{Ca}.

The integral cohomology of $K(\mathbb{Z}/n,2)$ in degree $\leq 8$ is isomorphic to the following graded commutative ring:
\begin{equation}\label{K(Z/2,2)}
\mathbb{Z}[\beta_n,Q_n,R_n,\rho_n]/(n\beta_n,\epsilon_{2}(n)\beta_n^2, \epsilon_{2}(n)nQ_n,\epsilon_{3}(n)nR_n,\epsilon_3(n)\rho_n),
\end{equation}
where $\operatorname{deg}(\beta_n)=3, \operatorname{deg}(Q_n)=5, \operatorname{deg}(R_n)=7$, and $\operatorname{deg}(\rho_n)=8$. In other words, each $H^i(K(\Z/n,2);\Z)$ for $3,5,6,7$, is cyclic (and possibly trivial) generated by, respectively, $\beta_n, Q_n, \beta_n^2, R_n$, of order
$$n,\epsilon_{2}(n)n,\epsilon_{2}(n),\epsilon_{3}(n)n,$$
and for $i=8$, a direct sum of two (possibly trivial) cyclic groups generated by $\beta_nQ_n$ and $\rho_n$, of order $\epsilon_2(n)$ and $\epsilon_3(n)$, respectively.
(See (2.5) of \cite{Gu}.)

For $n\geq3$, the ring $H^{*}(K(\mathbb{Z},n);\mathbb{Z})$ in degree $\leq n+3$ is isomorphic to the following graded rings:
\begin{equation}\label{K(Z,n)}
\begin{cases}
\mathbb{Z}[\iota_n,\Gamma_n]/(2\Gamma_n),\ n>3,\textrm{ even},\\
\mathbb{Z}[\iota_n,\Gamma_n]/(2\iota_n^2,2\Gamma_n),\ n>3,\textrm{ odd},\\
\mathbb{Z}[\iota_3,\Gamma_3]/(2\Gamma_3, \Gamma_3-\iota_3^2),\ n=3,
\end{cases}
\end{equation}
where $\iota_n$, of degree $n$, is the so-called fundamental class, and $\Gamma_n$, of degree $n+3$, is a class of order $2$.(See (2.1) of \cite{Gu}.)

\begin{proposition}\label{stable Gamma}
The classes $\Gamma_n\in H^{n+3}(K(\mathbb{Z},n);\mathbb{Z})$ as above for all $n\geq 3$, stabilize to the same stable cohomology operation $\operatorname{Sq}_{\mathbb{Z}}^3\in H^3(H\mathbb{Z};\mathbb{Z})$ of order $2$, where $HR$ denotes the Eilenberg - Mac Lane spectrum associated to a unit ring $R$. Moreover, the mod $2$ reduction of $\operatorname{Sq}_{\mathbb{Z}}^3$ is the well-understood Steenrod square $\operatorname{Sq}^3$. In other words, the following diagram in the homotopy category of spectra commutes:
\begin{equation}\label{spectra}
\begin{tikzcd}
H\mathbb{Z}\arrow{d}\arrow{r}{\operatorname{Sq}_{\mathbb{Z}}^3}&\Sigma^3H\mathbb{Z}\arrow{d}\\
H\mathbb{Z}/2\arrow{r}{\operatorname{Sq}^3}&\Sigma^3H\mathbb{Z}/2
\end{tikzcd}
\end{equation}
where the vertical arrows are the obvious ones.
\end{proposition}
For a proof see Lemma 2.1 of \cite{Gu}.

The author is grateful to Professor Christian Haesemeyer and Professor Diarmuid Crowley for many helpful conversations in the course of this work.

\section{the group $H^7(\mathbf{B}P(n,mn);\mathbb{Z})$}
As in the introduction we denote by $P(n,mn)$ the quotient group of the following inclusion of Lie groups:
\begin{equation*}
\mathbb{Z}/n\hookrightarrow SU_{mn}: t\mapsto e^{2\pi it/n}\mathbf{I}_{mn},
\end{equation*}
where $\mathbf{I}_r$ is the identity matrix of degree $r$, as well as its classifying space $\mathbf{B}P(n,mn)$. In this section we consider the cohomology group $H^7(\mathbf{B}P(n,mn);\mathbb{Z})$.

As shown in \cite{Gu}, the significance of the space $\mathbf{B}P(n,mn)$ is its role in the study of a topological Brauer class $\alpha$ of period $n$. Such a class may be decomposed as a summation
\[\alpha=\sum_p\alpha_p\]
where $p$ runs over the prime divisors of $n$ and the period of $\alpha_p$ is a power of $p$. To prove Proposition \ref{mainprop}, we only consider the case $4|n$. On the other hand, Theorem 3 of \cite{An2} asserts $\opn{ind}(\alpha)=\prod_p\opn{ind}(\alpha_p)$, with each $\opn{ind}(\alpha_p)$ a power of $p$. Hence it suffices to determine $\opn{ind}(\alpha_2)$.

In other words, we have reduced the problem to the case $n=2^r$, with $r>1$. According to Theorem \ref{last}, for $\operatorname{per}(\alpha)=2^r$, the maximal value of $\operatorname{ind}(\alpha)$ is either $2^{2r}$ or $2^{2r+1}$. In the rest of this paper we show that the former is true, which is the content of Proposition \ref{mainprop}.

Therefore, we assume the following for the rest of this paper unless otherwise specified:
\begin{convention}\label{convention}
	\[n=2^r,\ m=2^{2r},\ r>1.\]
\end{convention}

Let us remind ourselves of the following notation: for a simply connected space $X$, let $X[k]$ denote the $k$th level of the Postnikov tower of $X$. Consider the Postnikov tower of $\mathbf{B}P(n,mn)$ for $\epsilon_2(n)n|m$, $n>1$:
\begin{equation}\label{Postnikov}
\begin{tikzcd}
                        &\mathbf{B}P(n,mn)[6]\arrow{d}{=}\\
K(\mathbb{Z},4)\arrow{r}&\mathbf{B}P(n,mn)[5]\simeq K(\mathbb{Z}/n,2)\times K(\mathbb{Z},4)\arrow{d}\arrow{r}{\kappa_5}&K(\mathbb{Z},7)\\
&K(\mathbb{Z}/n,2)\arrow{r}{\kappa_3=0}&K(\mathbb{Z},5)
\end{tikzcd}
\end{equation}
where $\kappa_3$ and $\kappa_5$ are the k-invariant of the space $\mathbf{B}P(n,mn)$. The equation $\kappa_3=0$ follows from Proposition 4.11 of \cite{Gu}. Consequently, $\mathbf{B}P(n,mn)[5]$ is a trivial fibration over $K(\Z/n,2)$ with fiber $K(\Z,4)$.

Consider the canonical map
\[\mathbf{B}P(n,mn)\rightarrow\mathbf{B}P(n,mn)[6]\simeq K(\mathbb{Z}/n,2)\times K(\mathbb{Z},4).\]
For future reference we take notes of the induced homomorphism between integral cohomology rings. For the sake of consistency, we take notations from \cite{Gu}, despite their annoying looking, as follows:
\begin{equation}\label{bookkeeping}
\beta_n\times 1\mapsto x_1',\quad R_n\times 1\mapsto R_n(x_1'),\quad 1\times\iota_4\mapsto e_2',\quad 1\times\Gamma_4\mapsto\operatorname{Sq}_{\mathbb{Z}}^3(e_2'),
\end{equation}
where $x_1'$ and $e_2'$ are the additive generators of $H^3(\mathbf{B}P(n,mn);\mathbb{Z})\cong\mathbb{Z}/n$ and $H^4(\mathbf{B}P(n,mn);\mathbb{Z})\cong\mathbb{Z}$, respectively. Here $R_n$ is the generator of $H^7(K(\Z/n,2);\Z)$ as in (\ref{K(Z/2,2)}), regarded as a cohomology operation in the obvious way.

By the construction of Postnikov towers we have
\begin{lemma}\label{H7}
We have
\[H^7(\mathbf{B}P(n,mn);\mathbb{Z})\cong H^7(K(\mathbb{Z}/n,2)\times K(\mathbb{Z},4);\Z)/(\kappa_5).\]
In particular, the group $H^7(\mathbf{B}P(n,mn);\mathbb{Z})$ is generated by $R_n(x_1')$, $x_1'e_2'$, and  $\operatorname{Sq}_{\mathbb{Z}}^3(e_2')$.
\end{lemma}
Consider the short exact sequence of Lie groups
\[1\rightarrow\mathbb{Z}/n\rightarrow SU_{mn}\rightarrow P(n,mn)\rightarrow 1,\]
from which arises a fiber sequence
\[\mathbf{B}SU_{mn}\rightarrow\mathbf{B}P(n,mn)\rightarrow K(\mathbb{Z}/n,2)\]
considered at the beginning of Section 6 of \cite{Gu}. We denote the associated cohomological Serre spectral sequence by $(^SE_{*}^{*,*},{^Sd}_{*}^{*,*})$.

For $k>2$, it is well-known that
\[H^*(\mathbf{B}SU_k;\Z)=\mathbb{Z}[c_2,\cdots,c_k],\]
where $c_i$ is the $i$th universal Chern class of degree $2i$.
\begin{lemma}\label{Sd_5}
The differential $^Sd_5^{0,4}=0$. In particular, $c_2\in{^SE}_{2}^{0,4}$ is a permanent cocycle.
\end{lemma}
\begin{proof}
Diagram (\ref{Postnikov}) implies
\[H^5(\mathbf{B}P(n,mn);\mathbb{Z})\cong H^5(K(\mathbb{Z}/n,2);\mathbb{Z})\cong\mathbb{Z}/2n.\]
Hence we have $^Sd_5^{0,4}=0$. There is no other non-trivial differential out of $^SE_{*}^{0,4}$ for obvious degree reasons, so $c_2$ is a permanent cocycle.
\end{proof}

\begin{lemma}[Lemma 6.1, \cite{Gu}]\label{SE}
Recall that $H^{3}(K(\mathbb{Z}/n,2);\mathbb{Z}))\cong\mathbb{Z}/n$ is generated by an element $\beta_n$, and that $H^{7}(K(\mathbb{Z}/n,2);\mathbb{Z}))\cong\mathbb{Z}/n$ is generated by $R_n$. In the spectral sequence ${^{S}E}_{*}^{*,*}$, we have ${^{S}d}_{3}^{0,6}(c_3)=2c_{2}\beta_{n}$ with kernel generated by
\[\frac{n}{2}c_3,\]
and
\[{^{S}d}_{7}^{0,6}(\frac{n}{2}c_3)=nR_n.\]
All the other differentials out of ${^{S}E}_{*}^{0,6}$ are trivial.

\end{lemma}
See Figure \ref{^S E figure} for the differentials mentioned in the lemmas above.
\begin{figure}[!]
\begin{tikzpicture}
\draw[step=0.5cm,gray, thick];
\draw [ thick, <->] (0,6)--(0,0)--(7,0);

\node [below] at (3,0) {$3$};
\node [below] at (5,0) {$5$};
\node [below] at (7,0) {$7$};
\node [left] at (0,4) {$4$};
\node [left] at (0,6) {$6$};
\node [above right] at (0,4) {$\mathbb{Z}$};
\node [above right] at (0,6) {$\mathbb{Z}$};
\node [above right] at (3,0) {$\mathbb{Z}/n$};
\node [above right] at (5,0) {$\mathbb{Z}/\epsilon_{2}(n)n$};
\node [above right] at (7,0) {$\mathbb{Z}/\epsilon_{3}(n)n$};
\draw (0,4) [dashed, ->] to (5,0);
\draw (0,6) [thick, ->] to node [above] {$\times 2$}(3,4);
\draw (0,6) [dashed, ->] to (7,0);
\end{tikzpicture}
\caption{Low dimensional differentials of the spectral sequence $^{S}E_{*}^{*,*}$, when $\epsilon_3(n)n|m$, $n>1$. The dashed arrows represent trivial differentials.}\label{^S E figure}
\end{figure}
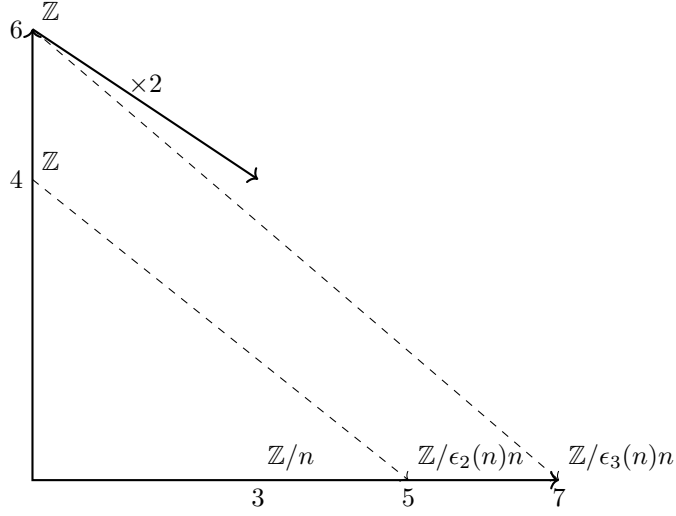

\begin{lemma}\label{cardH^7BP}
\begin{enumerate}
\item
The element $\operatorname{Sq}^3_{\mathbb{Z}}(e_2')$ is a linear combination of $x_1'e_2'$ and $R_n(x_1')$.
\item
The element $R_n(x_1')$ is of order $n$.
\item
The order of the group $H^7(\mathbf{B}P(n,mn);\mathbb{Z})$ is $2n$.
\end{enumerate}
\end{lemma}
\begin{proof}
As indicated in Figure \ref{^S E figure}, we have the exact sequence
\[0\rightarrow {^SE}_{\infty}^{7,0}\rightarrow H^7(\mathbf{B}P(n,mn);\mathbb{Z}) \rightarrow {^SE}_{\infty}^{3,4}\rightarrow 0,\]
where ${^SE}_{\infty}^{7,0}$ and ${^SE}_{\infty}^{3,4}$ are generated by $R_n(x_1')$ and $x_1'e_2'$, respectively. Hence (1) follows. The statement (2) is an immediate a consequence of Lemma \ref{SE}. To verify the statement (3), it suffices to check the orders of $^SE_{\infty}^{3,4}$ and $^SE_{\infty}^{7,0}$, which also follows from Lemma \ref{SE}.
\end{proof}


\section{the k-invariant $\kappa_5$}
Consider the space $\mathbf{B}P(n,mn)[6]$, the $6$th level of the Postnikov tower of $\mathbf{B}P(n,mn)$. Keep in mind that Convention \ref{convention} is still assumed.  In particular we have $2n|m$, in which case we have the homotopy equivalence (Proposition 4.11, \cite{Gu})
\begin{equation}\label{eq:directprod}
\mathbf{B}P(n,mn)[6]=\mathbf{B}P(n,mn)[5]\simeq K(\mathbb{Z}/n,2)\times K(\mathbb{Z},4).
\end{equation}

The goal of this section is to prove Proposition \ref{mainprop}, therefore Theorem \ref{main}, by determining the k-invariant
\[\kappa_5\in H^7(\mathbf{B}P(n,mn)[5];\Z)\cong H^7(K(\mathbb{Z}/n,2)\times K(\mathbb{Z},4);\Z).\]

The equations (\ref{K(Z/2,2)}) and (\ref{K(Z,n)}) together with the K{\"u}nneth formula give us
\begin{equation}\label{H^7(K*K)}
\begin{split}
&H^7(K(\mathbb{Z}/n,2)\times K(\mathbb{Z},4);\mathbb{Z})\\=&(R_n\times 1)\oplus (\beta_n\times\iota_4)\oplus (1\times\Gamma_4)\\
\cong&\mathbb{Z}/n\oplus\mathbb{Z}/n\oplus\mathbb{Z}/2,
\end{split}
\end{equation}
where $R_n\times 1$, $\beta_n\times\iota_4$ and $1\times\Gamma_4$ generate the three summands, respectively.

When $n$ is even, it follows from Theorem 6.8 of \cite{Gu} that up to an invertible coefficient, we have
\begin{equation}\label{k_51}
\kappa_5\equiv\lambda_1 R_n\times 1+ \lambda_2\beta_n\times\iota_4+1\times\Gamma_4\mod _2H^7(\mathbf{B}P(n,mn);\Z),
\end{equation}
an element in $H^7(K(\mathbb{Z}/n,2)\times K(\mathbb{Z},4);\mathbb{Z})$, where we have $\lambda_1=m\lambda/(2n)$ for some $\lambda$ invertible in $\mathbb{Z}/n$, $\lambda_2\in\Z/n$, and $_2H^7(\mathbf{B}P(n,mn);\Z)$ denotes the subgroup of $2$-torsions of $H^7(\mathbf{B}P(n,mn);\Z)$.


To narrow down the choices of $\kappa_5$, we have the following
\begin{lemma}\label{Gamma}
In $H^7(K(\mathbb{Z}/n,2)\times K(\mathbb{Z},4);\mathbb{Z})$ we have
\[\kappa_5\equiv 1\times\Gamma_4\mod (R_n\times1, \beta_n\times\iota_4).\]
\end{lemma}
\begin{proof}
Assume that the lemma is false. Since $1\times\Gamma_4$ is of order $2$, we have
\[\kappa_5\in (R_n\times1, \beta_n\times\iota_4).\]
Therefore, we have
\begin{equation*}
\begin{split}
&H^7(\mathbf{B}P(n,mn);\mathbb{Z})\\
\cong&H^7(K(\mathbb{Z}/n,2)\times K(\mathbb{Z},4);\mathbb{Z})/(\kappa_5)\\
\cong& (R_n\times 1)\oplus (\beta_n\times\iota_4)\oplus (1\times\Gamma_4)/(\kappa_5)\\
\cong& [(R_n\times 1)\oplus (\beta_n\times\iota_4)]/(\kappa_5)\oplus  (1\times\Gamma_4),
\end{split}
\end{equation*}
but this violates (1) of Lemma \ref{cardH^7BP}, which asserts that $H^7(\mathbf{B}P(n,mn);\mathbb{Z})$ is generated by
$R_n(x_1')$ and $\opn{Sq}^3_{\Z}(e_2')$.
\end{proof}





\begin{lemma}\label{lambda_2not2}
In $H^7(K(\Z/n,2)\times K(\Z,4);\Z)$, We have
\[\kappa_5\equiv\lambda_2\beta_n\times\iota_4\mod (R_n\times1, 1\times\Gamma_4),\]
where the coefficient $\lambda_2$ is invertible in $\mathbb{Z}/n$.
\end{lemma}
\begin{proof}
We argue by contradiction. Suppose that $\lambda_2$ is not invertible in $\mathbb{Z}/n$.
Notice that, for our choice of $m$ and $n$, equation (\ref{k_51}) implies
\[\kappa_5\equiv 2^{r-1}\lambda R_n\times 1+ \lambda_2\beta_n\times\iota_4+1\times\Gamma_4\mod _2H^7(\mathbf{B}P(n,mn);\Z).\]
Since $\lambda_2$ is not invertible, $\kappa_5$ has order less than $2^r$. On the other hand, it follows from (\ref{H^7(K*K)}) that the group $H^7(K(\mathbb{Z}/n,2)\times K(\mathbb{Z},4);\mathbb{Z})$ has order $2^{2r+1}$. Therefore the group
\[H^7(\mathbf{B}P(n,mn);\mathbb{Z})\cong H^7(K(\mathbb{Z}/n,2)\times K(\mathbb{Z},4);\mathbb{Z})/(\kappa_5)\]
has order greater than $2^{r+1}$, violating Lemma \ref{cardH^7BP}.
\end{proof}

Notice that $\kappa_5$ is determined by the Postnikov tower merely up to multiplication by an invertible coefficient. By the choice we have made of $n,m$, we are free to multiply $\kappa_5$ by any odd integer. Hence, we are enabled by Lemma \ref{lambda_2not2} to normalize (\ref{k_51}) by fixing $\lambda_2=1$:
\begin{equation*}
\kappa_5\equiv2^{r-1}\lambda R_n\times 1+\beta_n\times\iota_4+1\times\Gamma_4\mod _2H^7(\mathbf{B}P(n,mn);\Z)\cap (R_n\times 1, 1\times\Gamma_4).
\end{equation*}
where $\lambda$ is odd. However, since $R_n$ is of order $2^r$, and $2^{r-1}\equiv 2^{r-1}\lambda\mod 2^{r}$ for all odd integer $\lambda$, the preceding equation becomes
\begin{equation}\label{lambdainv}
\kappa_5=2^{r-1}R_n\times1+\beta_n\times\iota_4+1\times\Gamma_4\mod _2H^7(\mathbf{B}P(n,mn);\Z)\cap (R_n\times 1, 1\times\Gamma_4).
\end{equation}

Combining (\ref{lambdainv}) with Lemma \ref{Gamma}, we have
\begin{equation}\label{k_5}
\kappa_5\equiv2^{r-1}R_n\times 1+\beta_n\times\iota_4+1\times\Gamma_4\mod _2H^7(\mathbf{B}P(n,mn);\Z)\cap (R_n\times 1)
\end{equation}
Moreover, we have the following
\begin{lemma}\label{statement}
The abelian group $H^7(\mathbf{B}P(n,mn);\mathbb{Z})$ is additively generated by $R_n(x_1')$, $x_1'e_2'$ and $\operatorname{Sq}_{\mathbb{Z}}^3(e_2')$, modulo the relation
\begin{equation}\label{k_5'}
\mu R_n(x_1')+x_1'e_2'+\operatorname{Sq}_{\mathbb{Z}}^3(e_2')=0,
\end{equation}
where $\mu$ is either $0$ or $2^{r-1}$. Moreover, only one of the two possible relations holds.
\end{lemma}
\begin{proof}
The definition of $\kappa_5$ and \eqref{k_5} implies that at least one of the two cases holds. For the uniqueness, notice that if both relations hold, then we have
\[2^{r-1}R_n(x_1')=0\textrm{ and }x_1'e_2'+\operatorname{Sq}_{\mathbb{Z}}^3(e_2')=0.\]
Hence the abelian group $H^7(\mathbf{B}P(n,mn);\mathbb{Z})$ is generated by $R_n(x_1')$ and $\operatorname{Sq}_{\mathbb{Z}}^3(e_2')$, whose orders divide $2^{r-1}$ and $2$, respectively. Therefore the order of the group $H^7(\mathbf{B}P(n,mn);\mathbb{Z})$ divides $2^r$, violating Lemma \ref{cardH^7BP}.
\end{proof}


We turn to the hard work of determining $\mu$, for which we will need an auxiliary space $Y$ to be defined and studied as follows. Recall the homotopy equivalence (\ref{eq:directprod})
\[\mathbf{B}P(n,mn)[6]=\mathbf{B}P(n,mn)[5]\simeq K(\mathbb{Z}/n,2)\times K(\mathbb{Z},4).\]
Consider the following map
\[\mathbf{B}P(n,mn)\rightarrow\mathbf{B}P(n,mn)[6]\rightarrow K(\mathbb{Z},4)\]
such that both arrows above are the canonical maps. We may normalize this map so that it represents the generator $e_2'$ of $H^4(\mathbf{B}P(n,mn);\Z)$. We denote by $Y$ its homotopy fiber. In other words, we have a fiber sequence
\begin{equation}\label{fiber}
Y\rightarrow\mathbf{B}P(n,mn)\xrightarrow{e_2'} K(\mathbb{Z},4)
\end{equation}
where the second arrow is an additive generator $e_2'$ of $H^4(\mathbf{B}P(n,mn);\Z)$. The space $Y$ plays a key role in determining the coefficient $\mu$. By construction the second arrow in (\ref{fiber}) induces an isomorphism of homotopy groups
\[\pi_4(\mathbf{B}P(n,mn))\cong \pi_4(K(\mathbb{Z},4)).\]
This isomorphism lies in the long exact sequence of homotopy groups of the fiber sequence (\ref{fiber}), from which, together with (\ref{Bott}), we deduce
\begin{equation}\label{piY}
\pi_i(Y)\cong
\begin{cases}
\mathbb{Z}/n, \textrm{ if }n=2,\\
\mathbb{Z},\quad\textrm{if $6\leq i<2mn+1$, and $i$ is even},\\
0, \quad\textrm{if $0<i<2mn$, and $i$ is odd, or }i=4.\\
\end{cases}
\end{equation}

By taking the loop space of the last term of (\ref{fiber}) we obtain another fiber sequence
\begin{equation}\label{mainfiber}
K(\mathbb{Z},3)\xrightarrow{h}Y\rightarrow\mathbf{B}P(n,mn).
\end{equation}

Consider the canonical map from $Y$ to the bottom level of its Postnikov tower
\[g: Y\rightarrow K(\mathbb{Z}/n,2).\]
\begin{lemma}\label{H^6Y}
\begin{enumerate}[label=\textbf{Y.\arabic*}]
\item\label{item:g Y}
The induced homomorphisms
\[g^*: H^k(K(\mathbb{Z}/n,2);\mathbb{Z})\rightarrow H^k(Y;\mathbb{Z})\]
are isomorphisms for $k\leq 5$.
\item\label{item:H6Y}
\begin{equation*}
H^6(Y;\mathbb{Z}/2)=g^*(H^6(K(\mathbb{Z}/n,2);\mathbb{Z}/2))+ (\bar{\omega}),
\end{equation*}
where $\omega\in H^6(Y;\Z)$ generates an infinite cyclic group, and $\bar{\omega}$ is its reduction in cohomology with coefficients in $\mathbb{Z}/2$.
\end{enumerate}
\end{lemma}

\begin{proof}
It follows from (\ref{piY}) that we have the following partial picture of the Postnikov tower of Y:
\begin{equation*}
\begin{tikzcd}
                        &Y[7]\arrow{d}{=}\arrow{r}&K(\mathbb{Z},9)\\
K(\mathbb{Z},6)\arrow{r}&Y[6]\arrow{d}                      &\\
                        &Y[2]=K(\mathbb{Z}/n,2)\arrow{r}&K(\mathbb{Z},7)
\end{tikzcd}
\end{equation*}
The statements then follows from a simple observation on the fiber sequence
\[K(\mathbb{Z},6)\rightarrow Y[6]\rightarrow K(\mathbb{Z}/n,2)\]
and the Serre spectral sequence associated to it.
\end{proof}

Consider the fiber sequence (\ref{mainfiber}). Let $(E_{*}^{*,*}(\mathbb{Z}), d_{*}^{*,*})$ and $(E_{*}^{*,*}(\mathbb{Z}/2), \bar{d}_{*}^{*,*})$ denote the associated cohomological Serre spectral sequences with coefficients in $\mathbb{Z}$ and $\mathbb{Z}/2$, respectively.

We first consider the case with coefficients in $\mathbb{Z}/2$. This is easier since $\mu\equiv 0$ mod $2$ whatsoever. We study the homomorphism
\[h^*:H^6(Y;\mathbb{Z}/2)\rightarrow H^6(K(\mathbb{Z},3);\mathbb{Z}/2)\]
with the spectral sequence $E_{*}^{*,*}(\mathbb{Z}/2)$, from which, with a little luck, we obtain enough information on the homomorphism
\[h^*:H^6(Y;\mathbb{Z})\rightarrow H^6(K(\mathbb{Z},3);\mathbb{Z})\]
to determine a particular differential of the spectral sequence $E_{*}^{*,*}(\mathbb{Z})$, which in turn determines the coefficient $\mu$.

As in Lemma \ref{H^6Y}, we use overhead bars to denote the mod $2$ reductions of integral cohomology classes.

\begin{lemma}\label{lem:Z and Z/2}
\begin{enumerate}[label=\textbf{S.\arabic*}]
\item\label{item:iota}
The element
\[\iota_3\in H^3(K(\Z,3);\Z)\cong E_2^{0,3}(\Z)\]
is transgressive, satisfying
\[d_4^{0,3}(\iota_3)=\pm e_2',\]
and so is $\iota_3^2$:
\[d_7^{0,6}(\iota_3^2)=\operatorname{Sq}^3_Z(e_2').\]
\item\label{item:iota mod2}
The same holds in $E_*^{*,*}(\Z/2)$, i.e., $\bar{\iota}_3$ and $\bar{\iota_3}^2$ are both transgressive such that we have
\[\bar{d}_4^{0,3}(\bar{\iota}_3)=\pm \bar{e_2}',\]
and so is $\bar{\iota}_3^2$:
\[\bar{d}_7^{0,6}(\bar{\iota}_3^2)=\operatorname{Sq}^3(\bar{e_2}').\]
\item\label{item:quotients}
We have
\[E_4^{7,0}(\Z)=E_5^{7,0}(\Z)=E_6^{7,0}(\Z)\cong H^7(\mathbf{B}P(n,mn);\mathbb{Z})/(x_1'e_2')\]
and similarly
\[E_4^{7,0}(\Z/2)=E_5^{7,0}(\Z/2)=E_6^{7,0}(\Z/2)\cong H^7(\mathbf{B}P(n,mn);\mathbb{Z}/2)/(\bar{x}_1'\bar{e}_2').\]
\end{enumerate}
\end{lemma}
\begin{proof}
These follows from routine computations and may be read off from Figure \ref{Z/2 figure} and its integral analog. For \ref{item:iota} and \ref{item:iota mod2} we need the fact that transgressions commute with stable cohomology operations. (See \cite{Mc}.)
\end{proof}

\begin{lemma}\label{Z}
In the spectral sequence $E_{*}^{*,*}(\mathbb{Z})$, the class
\[\iota_3^2\in H^6(Y,\Z)\cong E_{2}^{0,6}(\Z)\]
is a permanent cocycle.
\end{lemma}
\begin{proof}
To show that $\iota_3^2$ is a permanent cocycle is to show that the homomorphism
\[h^*:H^6(Y;\mathbb{Z})\rightarrow H^6(K(\mathbb{Z},3);\mathbb{Z})\]
is surjective. Indeed, we show
\begin{equation}\label{rhoZ/2}
h^*(\omega)=\iota_3^2.
\end{equation}
Since $\iota_3^2$ is of order $2$, it suffices to show the mod $2$ version of (\ref{rhoZ/2}), i.e.,
\begin{equation}\label{rho mod2}
h^*(\bar{\omega})=\bar{\iota}_3^2.	
\end{equation}
It follows from Lemma \ref{statement} that we have
\begin{equation}\label{relationZ/2}
\bar{x}_1'\bar{e}_2'+\operatorname{Sq}^3(\bar{e}_2')=0\in H^7(\mathbf{B}P(n,mn);\mathbb{Z}/2).
\end{equation}
This relation, together with Lemma \ref{lem:Z and Z/2}, \ref{item:quotients}, shows
\[\operatorname{Sq}^3(\bar{e}'_2)\equiv0\in E_{5}^{7,0}(\mathbb{Z}/2)\cong H^7(\mathbf{B}P(n,mn);\mathbb{Z}/2)/(\bar{x}_1'\bar{e}_2').\]
Then it follows that $\bar{d}_{7}^{0,6}(\bar{\iota}_3^2)=0$, by Lemma \ref{lem:Z and Z/2}, \ref{item:iota mod2}. Hence, $\bar{\iota}_3^2$ is a permanent cocycle, which proves that $\bar{\iota}_3^2$ is in the image of $h^*$, and in particular, that $h^*$ is surjective in dimension $6$ and with coefficients in $\mathbb{Z}/2$.

To verify $h^*(\bar{\omega})=\bar{\iota}_3^2$, notice that $g$ factors as
$$Y\rightarrow\mathbf{B}P(n,mn)\rightarrow K(\mathbb{Z}/n,2)$$
since the map $Y\rightarrow\mathbf{B}P(n,mn)$ induces an isomorphism $\pi_2(Y)\cong\pi_2(\mathbf{B}P(n,mn))$. In particular, it follows that
\[g\circ h: K(\mathbb{Z},3)\rightarrow Y\rightarrow K(\mathbb{Z}/n,2)\]
is null homotopic. On the other hand, recall that Lemma \ref{H^6Y}, \ref{item:H6Y} asserts
\begin{equation*}
H^6(Y;\mathbb{Z}/2)=g^*(H^6(K(\mathbb{Z}/n,2);\mathbb{Z}/2))+ (\bar{\omega}),
\end{equation*}
so $h^*$ takes the first direct summand to $0$. It then follows from the surjectivity of $h^*$ that we have $h^*(\bar{\omega})=\bar{\iota}_3^2$, and we conclude that $\iota_3^2$ is a permanent cocycle.
\end{proof}

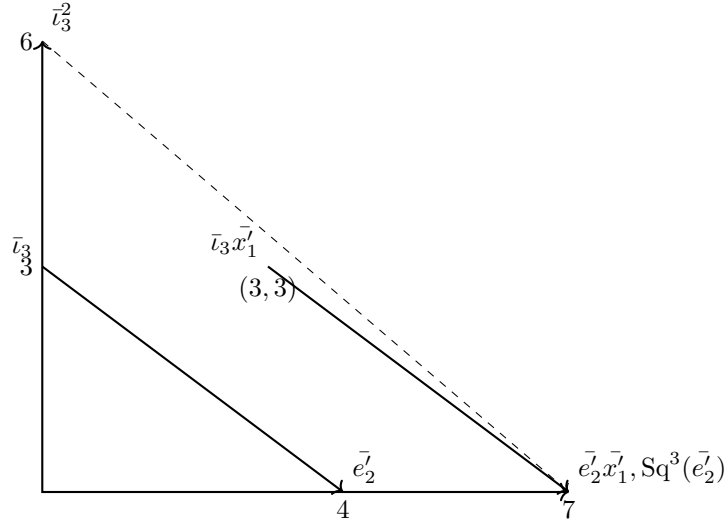
\begin{figure}[!]
\begin{tikzpicture}
\draw[step=0.5cm,gray, thick];
\draw [ thick, <->] (0,6)--(0,0)--(7,0);

\node [below] at (3,0) {$3$};
\node [below] at (4,0) {$4$};
\node [below] at (7,0) {$7$};
\node [below] at (3,3) {$(3,3)$};
\node [left] at (0,2) {$2$};
\node [left] at (0,3) {$3$};
\node [left] at (0,4) {$4$};
\node [left] at (0,5) {$5$};
\node [left] at (0,6) {$6$};
\node [above right] at (4,0) {$\bar{e_2'}$};
\node [above right] at (0,3) {$\bar{\iota}_3$};
\node [above left] at (3,3) {$\bar{\iota}_3\bar{x_1'}$};
\node [above right] at (7,0) {$\bar{e_2'}\bar{x_1'}, \operatorname{Sq}^3(\bar{e_2'})$};
\node [above right] at (0,6) {$\bar{\iota}_3^2$};
\shadedraw (0,0) rectangle (3,2);
\shadedraw (0,4) rectangle (3,5);
\draw (0,3) [thick, ->] to (4,0);
\draw (3,3) [thick, ->] to (7,0);
\draw (0,6) [dashed, ->] to (7,0);
\end{tikzpicture}
\caption{Differentials of the spectral sequence $E_{*}^{*,*}(\mathbb{Z}/2)$, i.e., the Serre spectral sequence associated to the fiber sequence
	\[K(\Z,3)\rightarrow Y\rightarrow\mathbf{B}P(n,mn).\]
The shaded areas represent some groups on the $E_2$ page that are $0$, and the dashed arrow represents a trivial differential.}\label{Z/2 figure}
\end{figure}

We proceed to determine $\kappa_5$.
\begin{proposition}\label{determinek_5}
Under Convention \ref{convention}, we have
\[\kappa_5=\beta_n\times\iota_4+1\times\Gamma_4.\]
\end{proposition}
\begin{proof}
By Lemma \ref{lem:Z and Z/2}, \ref{item:iota}, we have
\begin{equation}\label{iota_3^2}
d_7^{0,6}(\iota_3^2)=\operatorname{Sq}_{\mathbb{Z}}^3(e_2')=0\in E_6^{7,0}(\Z)\cong H^7(\mathbf{B}P(n,mn);\mathbb{Z})/(x_1'e_2').
\end{equation}
Therefore, by Lemma \ref{lem:Z and Z/2}, \ref{item:quotients}, we have
\begin{equation}\label{nu}
\nu x_1'e_2'+\operatorname{Sq}_{\mathbb{Z}}^3(e_2')=0\in H^7(\mathbf{B}P(n,mn);\mathbb{Z}).
\end{equation}
for $\nu=0$ or $\nu=1$. On the other hand, by (\ref{relationZ/2}) in the proof of Lemma \ref{Z}, we have
\begin{equation}\label{eq:Z relation}
x_1'e_2'+\operatorname{Sq}_{\mathbb{Z}}^3(e_2')\equiv 0\mod 2.
\end{equation}
Combining (\ref{nu}) and (\ref{eq:Z relation}) we have $\nu=1$, since otherwise we have $\bar{x}_1'\bar{e}_2'=0$, a contradiction. By Lemma \ref{statement}, we may now conclude.
\end{proof}

\begin{proof}[Proof of Proposition \ref{mainprop} and Theorem \ref{main}] Let $X=\opn{sk}_8(K(\Z/n,2))$ be the $8$-skeleton of $K(\Z/n,2)$, and $\alpha\in H^3(X;\Z)$ the restriction of the canonical generator $\beta_n\in H^3(K(\Z/n,2);\Z)$. Recall that it suffices to work under Convention \ref{convention}:
\[n=2^r,\ m=2^{2r}\ r>1.\]
It follows from Theorem \ref{last} that $\opn{ind}(\alpha)$ is either $2^{3r}$ or $2^{3r+1}$. It suffices to show $\opn{ind}(\alpha)=2^{2r}$, for which we proceed to study the homotopy lifting problem discussed in the introduction (\ref{lift'}):
\begin{equation}\label{lifting}
\begin{tikzcd}
                                        &\mathbf{B}P(2^r,2^{3r})\arrow{d}\\
X\arrow[r,hook]\arrow[ur,dashrightarrow]&K(\mathbb{Z}/2^r,2)
\end{tikzcd}
\end{equation}
where the bottom arrow is the obvious inclusion. By construction, this diagram is equivalent to
\begin{equation}\label{lifting1}
\begin{tikzcd}
                                                      &\mathbf{B}P(2^r,2^{3r})[7]\arrow{d}\\
K(\mathbb{Z}/2^r,2)\arrow{r}{=}\arrow[ur,dashrightarrow]&K(\mathbb{Z}/2^r,2)
\end{tikzcd}
\end{equation}
By (\ref{eq:directprod}), we have the following
\begin{equation}\label{lifting2}
\begin{tikzcd}
                                                      &\mathbf{B}P(2^r,2^{3r})[7]\arrow{d}{=}&\\
                                                      &\mathbf{B}P(2^r,2^{3r})[6]\arrow{d}&\\
                                                      &\mathbf{B}P(2^r,2^{3r})[5]\simeq K(\mathbb{Z}/2^r,2)\times K(\mathbb{Z},4)\arrow{d}\arrow{r}{\kappa_5}&K(\mathbb{Z},7)\\
K(\mathbb{Z}/n,2^r)\arrow{r}{f_3=\operatorname{Id}}\arrow{ur}{f_5}\arrow[uur,bend left,dashrightarrow]&K(\mathbb{Z}/2^r,2)&
\end{tikzcd}
\end{equation}
where the map $f_5$ is the obvious inclusion. Therefore $f_5^*$ annihilates all cohomology classes of $K(\mathbb{Z},4)$ in positive degrees, in particular, $\iota_4$ and $\Gamma_4$. Therefore, by Proposition \ref{determinek_5}, we have
\[f_5^*(\kappa_5)=f_5^*(\beta_{2^r}\times\iota_4+1\times\Gamma_4)=0,\]
and the dashed arrow in (\ref{lifting2}) exists. Therefore, by Proposition \ref{lift to BP} we have $\opn{ind}(\alpha)\neq 2^{3r+1}$, and we have proved Proposition \ref{mainprop}.

It remains to show the divisibility relations (\ref{bound1}). For an arbitrary finite CW complex $X$ of dimension $8$ and $\alpha\in H^3(X;\Z)_{\textrm{tor}}$ of period $n$, take $\alpha'\in H^2(X;\Z/n)$ such that the Bockstein homomorphism takes $\alpha'$ to $\alpha$. Then $\alpha'$, up to homotopy, gives rise to a map $X\rightarrow K(\Z/n,2)$. Now apply Proposition \ref{mainprop} and compare the twisted AHSSs of the pairs $(X,\alpha)$ and $(K(\Z/n,2),\beta_n)$. Theorem \ref{main} follows from the naturality of the twisted AHSSs together with Theorem \ref{AH diff}.
\end{proof}
\bibliographystyle{abbrv}
\bibliography{tpip8IIref}


\end{document}